\documentclass[10pt, reqno]{amsart}
\usepackage{amsmath,amssymb,amsbsy,amsfonts,latexsym,amsthm,amscd,amsxtra, xfrac,
amssymb, marginnote}

\usepackage{amsfonts,amsmath,amssymb,amsthm,amscd,url}
\usepackage[final]{showkeys}
\usepackage{xspace}
\usepackage[matrix,arrow]{xy}
\usepackage{mathrsfs,eucal,enumerate,mathtools}
\usepackage{epsfig,psfrag}
\usepackage{graphicx}
\usepackage{color}
\usepackage[hyperindex=true, 
					colorlinks=false]{hyperref}

\newtheorem{theorem}{Theorem}
\newtheorem{lemma}[theorem]{Lemma}
\newtheorem{proposition}[theorem]{Proposition}
\newtheorem{corollary}[theorem]{Corollary}

\newtheorem{definition}{Definition}[section]

\newtheorem*{remark}{Remark}

\newtheorem*{Pf}{Proof}

\newcommand\beqa[1]{ \begin{eqnarray} \label{#1}}
\newcommand{\eeqa}{ \end{eqnarray} }
\newcommand{\beqano}{ \begin{eqnarray*} }
\newcommand{\eeqano}{ \end{eqnarray*} }

\newcommand{\R}{ {\mathbb R}   }
\newcommand{\Z}{ {\mathbb Z}   }
\newcommand{\C}{ {\mathbb C}   }

\newcommand{\defeq}{\coloneqq} 

\newcommand{\CC}{\mathbb{C}}       
\newcommand{\NN}{\mathbb{N}}  
\newcommand{\QQ}{\mathbb{Q}}       
\newcommand{\RR}{\mathbb{R}}       
\newcommand{\ZZ}{\mathbb{Z}}       

\newcommand{\TT}{\mathbb{T}}

\newcommand{\PP}{\widehat{\mathbb{C}}}

\newcommand{\ph}{\varphi}          

\newcommand{\be}{\beta}

\newcommand{\sig}{\sigma}
\renewcommand{\th}{\theta}

\newcommand{\ze}{\zeta}
\newcommand{\om}{\omega}
\newcommand{\Om}{\Omega}

\newcommand{\gC}{\mathscr C}       
\newcommand{\gH}{\mathscr H}       

\newcommand{\gO}{\mathscr O}

\newcommand{\col}{\colon}          

\newcommand{\demi}{\frac{1}{2}}
\newcommand{\dem}{\tfrac{1}{2}}
\newcommand{\ti}{\tilde}
\newcommand{\ens}{\enspace}
\newcommand{\IM}{\mathop{\Im m}\nolimits}
\newcommand{\RE}{\mathop{\Re e}\nolimits}
\newcommand{\ie}{\emph{i.e.}\ }

\newcommand{\cf}{\emph{cf.}}

\newcommand{\Imp}{\quad\Longrightarrow\quad}

\newcommand{\bc}{\beta_{_\CC}}
\newcommand{\ioi}{\int_0^1}
\newcommand{\thom}{(\theta, \omega)}

\newcommand{\tu}{\tilde{u}}

\newcommand{\dst}{\displaystyle}

\newcommand{\ov}{\overline}

\DeclarePairedDelimiter\abs{\lvert}{\rvert}%
\DeclarePairedDelimiter\norm{\lVert}{\rVert}%

\newcommand{\IN}{^{\textnormal{(i)}}}
\newcommand{\EX}{^{\textnormal{(e)}}}
\newcommand{\Chol}{\gC^1_{\textrm{hol}}}

\newcommand{\CK}{{\gC^1_{\textrm{hol}}(K,B)}}

\newcommand{\beglab}[1]{\begin{equation}	\label{#1}}
\newcommand{\edla}{\end{equation}}

\def\beal{\begin{aligned}}
\def\enal{\end{aligned}}

\renewcommand \th {\theta}

\renewcommand \phi {\varphi}

\def\om{\omega}
\def\~{\tilde}
\def\Bbb{\mathbb}

\def\R{\Bbb R}
\def\C{\Bbb C}
\def\Z{\Bbb Z}

\def \S {{\mathbb S}}

\def\th{\theta}

\begin{document}

\title{On the regularity of Mather's $\be$-function for standard-like twist maps}
\author[Carminati]{Carlo Carminati}
\address{Dipartimento di Matematica, Universit\`a di Pisa, Largo Bruno Pontecorvo 5, 56127 Pisa, Italy.}
\email{carlo.carminati@unipi.it}
\author[Marmi]{Stefano Marmi}
\address{Scuola Normale Superiore, CNRS Fibonacci Laboratory, Piazza dei Cavalieri 7, 56126 Pisa, Italy.}
\email{s.marmi@sns.it}
\author[Sauzin]{David Sauzin}
\address{CNRS IMCCE Paris Observatory
PSL University, 77 av. Denfert-Rochereau 75014 Paris, France.}
\email{david.sauzin@obspm.fr }
\author[Sorrentino]{Alfonso Sorrentino}
\address{Dipartimento di Matematica, Universit\`a degli Studi di Roma ``Tor Vergata'', Via della ricerca scientifica 1, 00133 Rome, Italy.}
\email{sorrentino@mat.uniroma2.it}

\date{\today}
\subjclass[2010]{}

\maketitle

\begin{abstract}
  We consider the minimal average action (Mather's $\be$ function)
  for area preserving twist maps of the annulus.  
  The regularity properties of this function share interesting
  relations with the dynamics of the system. 
  We prove that 
  the $\be$-function associated to a standard-like twist map admits
  a 
  unique $\gC^1$-holomorphic complex extension, which coincides
  with this function on the set of real diophantine frequencies.
\end{abstract}

\section{Introduction}

In this note we would like to investigate some regularity properties of the so-called  {\it Mather's $\be$-function} (or {\it minimal average action}) for twist maps of the annulus.
This object is related to the minimal {average} action of configurations
with a prescribed rotation number (the so-called {\it Aubry-Mather
  orbits}) and plays a crucial role in the study of the dynamics of twist
maps; see section \ref{AMtheory} for a more detailed
introduction.
In particular,  many intriguing  questions and conjectures related to  problems in dynamics, analysis and geometry have been translated into questions
 {about} this function and its regularity {properties} (see for
 example \cite{MS, Siburg, Sor14, SVeselov, SV} and references
 therein), shedding a new light on these issues and, in some cases, paving the way for their solution. \\

{Two of the main questions that underpin our current interest in the subject are the following:}

\noindent a) {\it Do  regularity properties of  $\be$-function (\ie differentiability, higher smoothness, etc.) allow one to infer any information on the dynamics of the system?}\\
%
%
b) {\it  To which extent does this function identify the system? Does it satisfy any sort of rigidity property?}\\

Despite the huge amount of attention that these questions  have attracted over the past years---in particular, understanding its regularity and its implications---they remain essentially open. 
In the twist map case, the best result known is that this map is strictly convex and differentiable at all irrationals. Moreover, differentiability at a rational number $p/q$ is a very atypical phenomenon: 
it corresponds to the existence of an invariant circle consisting of periodic orbits  whose rotation number is $p/q$ (see \cite{Mather90}).
An extension of these results to surfaces was provided in \cite{MS}.

Goal of this article is to address this regularity issue and provide
some new interesting answers in the special case of {\it standard-like
  maps}.  More specifically, our starting point is the paper
\cite{CMS} which establishes some
rigidity properties of the complex extension of analytic
parametrizations of KAM curves. We use the main result of \cite{CMS}
to build up a $\gC^1$-holomorphic complex function which coincides
with Mather's $\be$ function on the set of real diophantine
frequencies and we prove that this extension is unique. {See
  Theorem \ref{th:main} and Corollary \ref{cor:betaquasianalytic} for
  precise statements.}\\ 

{The article is organized as follows. In section \ref{AMtheory} we provide a brief introduction to Aubry-Mather theory and introduce the main object of investigation (Definition \ref{defbeta}). In section \ref{secmainresult} we state our main results (Theorem \ref{th:main} and Corollary \ref{cor:betaquasianalytic}), whose proofs will be detailed in section \ref{sec:pfThmmain}. Some auxiliary results will be described in section \ref{sec:intermed} and appendix \ref{appendix}.}\\ 

\subsection*{Acknowledgements}
 The authors acknowledge the support of the Centro di Ricerca Matematica Ennio de Giorgi and of UniCredit Bank R$\&$D group for financial support through the {\it Dynamics and Information Theory Institute} at the Scuola Normale Superiore. 
{CC, SM and AS acknowledge the support of MIUR PRIN Project {\it Regular and stochastic behaviour in dynamical systems} nr. 2017S35EHN.
AS acknowledges the support of the MIUR Department of Excellence grant
CUP E83C18000100006.}
DS thanks Fibonacci Laboratory for their hospitality.
CC has been partially supported by the GNAMPA group of the “Istituto Nazionale di Alta Matematica” (INdAM).
  \\

\section{ A Synopsis of Aubry--Mather theory for twist maps of the cylinder} \label{AMtheory}

At the beginning of 1980s Serge Aubry and John Mather developed, independently, a novel {and fruitful} approach to the study of monotone twist maps of the annulus, based on the so-called {\it principle of least action}, nowadays  commonly called {\it Aubry--Mather theory}.
{They} pointed out the existence of  {\it global action-minimizing orbits} for any given rotation number; these orbits  minimize the discrete Lagrangian action with fixed end-points  on all time intervals (for a more detailed introduction,  see  for example \cite{Bangert, MatherForni, Siburg, SorLecNotes}).\\

Let us consider the annulus $ \S^1 \times (a,b)$, where $\S^1:=\R/\Z$
and $a,b\in [-\infty,+\infty]$.  Let us consider a diffeomorphism
$f: \S^1 \times (a,b) \longrightarrow \S^1 \times (a,b)$ and its lift
to the universal cover $\R\times (a,b)$, that we will continue to
denote by $f$; we assume that ${f}(x+1,y) = {f}(x,y) + (1,0)$ for each
$(x,y) \in \R \times (a,b)$. 

In the case in which $a,b$ are both finite, we will assume that ${f}$ extends continuously to $\R\times [a,b]$ and that it preserves the boundaries, with the corresponding dynamics being rotations by some fixed angles $\omega_\pm$:
\begin{equation}\label{extensionrotation}
{f}(x,a) = (x+\omega_-,\; a) \quad {\rm and} \quad {f}(x,b) = (x+\omega_+, \;b).
\end{equation}
For simplicity, we set $\omega_{\pm}=\pm \infty$ if $a=-\infty$ or $b=+\infty$.\\

\begin{definition}
A map 
\begin{eqnarray*}
f: \R \times (a,b) &\longrightarrow& \R \times (a,b) \\
(x_0,y_0) &\longmapsto& (x_1,y_1)
\end{eqnarray*}
is called a {\em monotone twist map} if:
\begin{itemize}
\item[(i)] $f(x_0+1, y_0) = {f}(x_0, y_0) + (1,0)$;
\item[(ii)] $f$ preserves orientation and the boundaries of
  $\R\times (a,b)$, \ie 
$y_1(x_0,y_0) \xrightarrow[y_0\to a]{} a$ and
$y_1(x_0,y_0) \xrightarrow[y_0\to b]{} b$ uniformly in~$x_0$;
%
\item[(iii)] if $a$ or $b$ is finite, then $f$ can be continuously extended to the boundary by a rotation, as in (\ref{extensionrotation});
\item[(iv)] $f$ satisfies the {\it monotone twist
    condition}\footnote{%
    The twist condition can be geometrically described by saying that
    each vertical $\{x=x_0\}$ is mapped by~$f$ to a graph over the
    $x$-axis. In particular, for each $x_0$ and $x_1$, there exists a
    unique $y_1$ such that $(x_1,y_1)$ belongs to the image of $\{x=x_0\}$.}
$$\frac{\partial x_1}{\partial y_0} (x_0,y_0)>0 \quad 
\text{for all $(x_0,y_0)\in \R\times (a,b)$;}
$$  
\item[(v)] $f$ is {\it exact symplectic}, \ie 
there exists a function $h\col \R\times\R \to \R$
%
%
such that $h(x_0+m,x_1+m)=h(x_0,x_1)$ for all $m\in \ZZ$ and
$$
y_1\,dx_1 - y_0\,dx_0 = dh(x_0,x_1).
$$
\end{itemize}
The interval $(\omega_-, \omega_+)\subset \R$ is then called the {\it twist
  interval} of $f$
and any function~$h$ as above is called a generating function for~$f$.\\
\end{definition}

\begin{remark}
{Observe that  (iv) implies that one can use
  $(x_0,x_1)$ as independent variables instead of $(x_0,y_0)$, namely 
 if $(x_1,y_1) = f(x_0,y_0)$ then $y_0$ is uniquely determined. Moreover,}
the generating function $h$ allows one to reconstruct completely the dynamics of $f$; in fact,
it follows from {property~(v)} 
that:
\begin{equation} \label{genfuncttwistmap}
\left\{
\begin{aligned}
y_1 &= \tfrac{\partial h}{\partial x_1}(x_0,x_1)\\[1ex]
y_0 &= - \tfrac{\partial h}{\partial x_0}(x_0,x_1).
\end{aligned}
\right.
\end{equation}
{Observe that condition (iv) corresponds to asking that
$$
\frac{\partial^2 h}{\partial x_0 \partial x_1} <0.
$$
}
\end{remark}

\noindent{\bf Examples.} 
\begin{itemize}
\item[1.] The easiest example is the following (which is an example of {\it integrable} twist map):
$$
f(x_0,y_0)= (x_0 + \rho(y_0), y_0),
$$
where $\rho: (a,b) \longrightarrow \R$ and, in order to satisfy the
twist condition, it is strictly increasing, \ie $\rho'(y_0)>0$
for each $y_0\in (a,b)$. The dynamics is very easy: the space is
foliated by a family of invariant straight lines  $\{y=y_0\}$, on
which the dynamics is a translation by $\rho (y_0)$. Observe that if we
look at the projected map on the annulus $\S^1\times (a,b)$, we obtain
a family of invariant circles $\{y=y_0\}$ on which the map acts as a
rotation by $\rho(y_0).$ 

It is easy to check that a generating function 
is given by $h(x_0,x_1)= \sig(x_1-x_0)$ with any~$\sig$ such that
$\sig'$ is the inverse bijection of~$\rho$.\\
\item [2.] The {\it standard maps}. 
One of the simplest (yet, very challenging) non-integrable  twist map is the so-called {\it standard map} {(this name appeared for the first time in \cite{Chir})}:
$$
f_\varepsilon(x_0,y_0)=(x_1,y_1) \ens\text{with}\ens \left\{
\begin{aligned}
x_1 &= x_0+y_0+ \varepsilon \sin (2\pi x_0)\\
y_1 &= y_0+ \varepsilon \sin (2\pi x_0)
\end{aligned}
\right.
$$
where $\varepsilon >0$ is a parameter ($\varepsilon=0$ would correspond to an integrable map). 
It is easy to check that a generating function is given by 
$$
h_\varepsilon (x_0,x_1) = \frac{1}{2} (x_1-x_0)^2 - \frac{\varepsilon}{2\pi}  \cos (2\pi x_0).
$$
This map has been the subject of extensive investigation, both from an
analytical and numerical points of view. 
An interesting question
concerns what happens in the transition between integrability and
chaos; in particular, {can one determine} at which value of~$\varepsilon$ an invariant curve of
a given rotation number breaks down, or at which value  there are no
more invariant curves? {See for example \cite{Chir, MMP, MP, Mather84, MaS}} (although the
literature on the topics is vast). 

In section \ref{secmainresult} we will focus on a generalized version
of this map (see \eqref{eq:twist}), namely:
$$
T_g(x,y)=(x',y') \ens\text{with}\ens \left\{
\begin{aligned}
x'&=x+y+g(x) \\
y'&=y+g(x)
\end{aligned}
\right.
$$
with $g$ a 1-periodic, real analytic function of zero mean. {We will refer to this kind of map as {\it standard-like twist map}.}\\

\item [3.]  Another interesting example is provided by {\it Birkhoff
    billiards}. This dynamical model describes the motion of a point
  inside a planar strictly convex domain $\Omega$ with smooth
  boundary. The billiard ball moves with unit velocity and without
  friction following a rectilinear path; when it hits the boundary it
  reflects according to the standard reflection law: the angle of
  reflection is equal to the angle of incidence.  See \cite{Tabach}
  for a more detailed introduction. 

If one considers the arc-length parametrization of the boundary
$\partial \Omega$, then one can describe the billiard map as a map $B(s_0, -\cos(\varphi_0))=(s_1,-\cos(\varphi_1))$, where $s_{0,1}$ refer to the starting and hitting point on the boundary, while 
$\phi_{0,1} \in (0,\pi)$ are the starting and hitting directions of the trajectory, with respect to the positive tangent directions on the boundary. With respect to these coordinates $(x=s,y=-\cos\varphi)$ the billiard map is a monotone twist map.\\

\item [4.]  Let us consider 
\begin{eqnarray*}
H: \S^1\times \R \times \S^1 &\longrightarrow& \R\\
(x,y,t) &\longmapsto& H(x,y,t),
\end{eqnarray*}
a $C^2$ Hamiltonian which is strictly convex and superlinear in the
momentum variable (\ie $\partial^2_y H >0$ and 
$\lim_{|y|\rightarrow +\infty} \frac{H(x,y)}{|y|}=+\infty$); 
then its time-1 map flow $\Phi_H^1: S^1\times \R\longrightarrow S^1\times \R $ can be lifted to a monotone twist map on $\R\times \R$.
Such Hamiltonians are often called  {\it Tonelli Hamiltonian}; see
\cite{SorLecNotes}. 

Moser in \cite{Moser} proved that every twist diffeomorphism is the time one map associated to a suitable 
Tonelli Hamiltonian system.\\
\end{itemize}

\vspace{10 pt}

As follows from (\ref{genfuncttwistmap}), any orbit $\{(x_i,y_i)\}_{i\in\Z}$ of the monotone twist diffeomorphism $f$ is completely determined by the sequence $(x_i)_{i\in \Z}$. Moreover, this sequence corresponds to critical points  of the discrete {\it action functional}:
$$
\R^\Z \ni (x_i)_{i\in\Z} \longmapsto \sum_{i\in \Z} h(x_i, x_{i+1}),
$$
{where the series is to be interpreted as a formal object.}
{This means that $(x_i)_{i\in \Z}$ comes from an
  orbit of $f$ if and only if}
$$
\partial_2 h (x_{i-1},x_{i}) + \partial_1 h(x_i, x_{i+1}) =0 \qquad \mbox{ for all}\; i\in \Z
$$
(hereafter we will denote by $\partial_j$ the derivative with respect to the $j$-th component).\\


Observe that while orbits correspond to critical points of the
action-functional, yet they are not in general minima\footnote{The
  concept of minimum might seem quite ambiguous in this setting, since
  the action-functional is generally a divergent series. Here---as is generally
  done in similar contexts in {classical and} statistical
  mechanics---by {\it minimum} we mean that each subsequence of finite
  length minimizes the action functional among all configurations with
  the same end-points and the same length.}.  {\it Aubry-Mather
  theory} is concerned with the study of orbits that minimize this
action-functional amongst all configurations with a prescribed
rotation number; we will call these orbits {\it action-minimizing
  orbits} or, simply, {\it minimizers}.
We will call the corresponding sequences $(x_i)_{i\in\Z}$ {\it minimal configurations}.\\  

Recall that the rotation
number of an orbit $\{(x_i,y_i)\}_{i\in\Z}$ is given by
$\omega = \lim_{|i|\rightarrow \pm \infty} \frac{x_i}{|i|}$, if this
limit exists.  For example, in example 1 above, orbits starting at
$(x_0,y_0)$ have rotation number $\rho(y_0).$
%
A natural question is then: does $f$ admit orbits with any prescribed rotation number? 
In \cite{Birkhoff}, Birkhoff proved that for every rational number $p/q$ in the twist interval $(\omega_-, \omega_+)$, there exist at least two periodic orbits of $f$ with rotation number $p/q$.\\

In the eighties, Aubry \cite{Aubry} and Mather \cite{Mather82} generalised independently this result to irrational rotation numbers. More precisely:\\

\noindent {\bf Theorem (Aubry, Mather).} {\it A monotone twist map possesses action-minimizing orbits for every rotation number in its twist interval $(\omega_-,\omega_+)$}.\\

\begin{remark}
They also showed that every action-minimizing orbit lies on a
Lipschitz graph over the $x$-axis and that if there exists an
invariant circle, then every orbit on that circle is a minimizer. Hence, in the integrable case (see Example 1), each orbit is a minimizer.
In a naive---yet meaningful---way, action-minimizing orbits ``resemble'' (and generalise) motions on invariant circles, even in the case in which invariant circles do not exist.\\
\end{remark}

\smallskip

Two very important objects in the study of these action-minimizing orbits are
represented by the so-called  Mather's {\it minimal average actions},
also called $\alpha$ and $\be$-functions:  in some sense they can be seen as an integrable Hamiltonian and Lagrangian associated to the system. \\

Let us now introduce the {\it minimal average action} (or {\it Mather's $\be$-function}) more precisely.

\begin{definition}\label{defbeta}
Given $\om \in (\om_-,\om_+)$, let $x^{\omega} = (x_i)_{i\in\Z}$ be any minimal configuration with rotation number $\omega$. Then, the value of the {minimal average action} at $\omega$ is given by
\begin{equation}\label{avaction}
\be(\omega) = \lim_{ \substack{N_1\to-\infty \\ N_2 \to +\infty}} %
\, \frac{1}{N_2-N_1} \, \sum_{i=N_1}^{N_2-1} h(x_i,x_{i+1}).
%
%
\end{equation}
This value is well-defined, since 
{the limit exists and}
does not depend on the chosen orbit.\\
\end{definition}

This function $\be: (\omega_-,\omega_+)\longrightarrow \R$ encodes a lot of interesting information on the dynamical and topological properties of these action-minimizing orbits and the system. In particular, 
understanding whether or not this function is differentiable, or even smoother, and what are the implications of its regularity to the dynamics of the system has revealed to be a central question in the study of twist maps and, more generally, of Tonelli Hamiltonian systems (see for example \cite{Mather90, MS}).
While for higher dimensional system this question represents a formidable problem (and is still quite far from being completely understood), in the twist-map case \cite{Mather90} (and for surfaces, see \cite{MS}) the situation is much more clear. In fact:
\begin{itemize}
\item[i)] $\be$ is strictly convex and, hence, continuous (see \cite{MatherForni});
\item[ii)] $\be$ is differentiable at all irrationals (see \cite{Mather90});
\item[iii)] $\be$ is differentiable at a rational $p/q$ if and only if there exists an invariant circle consisting of periodic aaction-minimizing orbits of rotation number $p/q$ (see \cite{Mather90}).\\
\end{itemize}

In particular, being $\be$ a convex function, one can consider its convex conjugate:
$$
\alpha( c ) = \sup_{\omega\in \R} \left[ \omega \, c - \be(\omega)\right].
$$

This function---which is generally called {\it Mather's
  $\alpha$-function}---also plays an important r\^ole in the study of
action-minimizing orbits and in Mather's theory (particularly in higher dimension, see for example  \cite{MS, SV}). We
refer interested readers to surveys \cite{MatherForni, Siburg, SorLecNotes}.\\

Observe that for each $\omega$ and $c$ one has:
$$
\alpha( c ) + \be( \omega) \geq \omega  c,
$$
where equality is achieved if and only if $c\in \partial
\be(\omega)$ or, equivalently, if and only if  $\omega \in \partial
\alpha (c )$; the symbol $\partial$ denotes in this case the set of
subderivatives of the function---meant as the slopes of supporting
lines at a point---which is always non-empty, and is a singleton if
and only if the function is differentiable at that point.\\

\begin{remark}
  In the billiard case, since a generating function of the billiard
  map is minus the Euclidean distance, $-\ell$, the action of an orbit
  coincides up to sign to the length of the trajectory that
  the ball traces on the table $\Omega$; hence, minimizing the action
  corresponds to maximizing the total length.  Therefore, for rational
  numbers $-q \be (p/q)$ represents the maximal perimeter of
  polygons of type $(p,q)$ (\ie\!\!, roughly speaking, polygons with $q$
  vertices and winding number~$p$). Moreover, it is possible to
  express many interesting invariants of billiards in terms of these
  functions (see also \cite{Sor14}):
\begin{itemize}
 \item If $\Gamma_{\omega}$ is a caustic with rotation number $\omega \in (0,1/2]$, then $\be$ is differentiable at $\omega$ and $\be'(\omega)= - {\rm length}(\Gamma_{\omega}) =: -|\Gamma_{\omega}|$ (see \cite[Theorem 3.2.10]{Siburg}). In particular,   $\be$ is always differentiable at $0$ and $\be'(0)= - |\partial \Omega|$.
 \item If $\Gamma_{\omega}$ is a caustic with rotation number $\omega \in (0,1/2]$, then one can associate to it another invariant, the so-called {\it Lazutkin invariant} $Q(\Gamma_{\omega})$.  More precisely 
\begin{equation}\label{lazinv}
Q(\Gamma_{\omega}) = |A-P| + |B-P| - |\stackrel \frown {AB} |
\end{equation}
{where $P$ is any point on  $\partial \Omega$, $A$ and $B$ are the corresponding points on  $\Gamma_\omega$ at which the half-lines exiting from $P$ are tangent to $\Gamma_\omega$ (see figure \ref{figureLazutkin})}, 
 and $|\cdot |$ denotes the euclidean length and $|\stackrel \frown
 {AB}|$ the length of the arc on the caustic joining $A$ to $B$.
 This quantity is connected to the value of the $\alpha$-function  (see \cite[Theorem 3.2.10]{Siburg}):
 $$Q(\Gamma_\omega) = \alpha(  \be'(\omega)  ) = \alpha( - |\Gamma_{\omega}|).\\$$
 \end{itemize}
 \end{remark}
\begin{figure} [h!]
\begin{center}
\epsfig{file=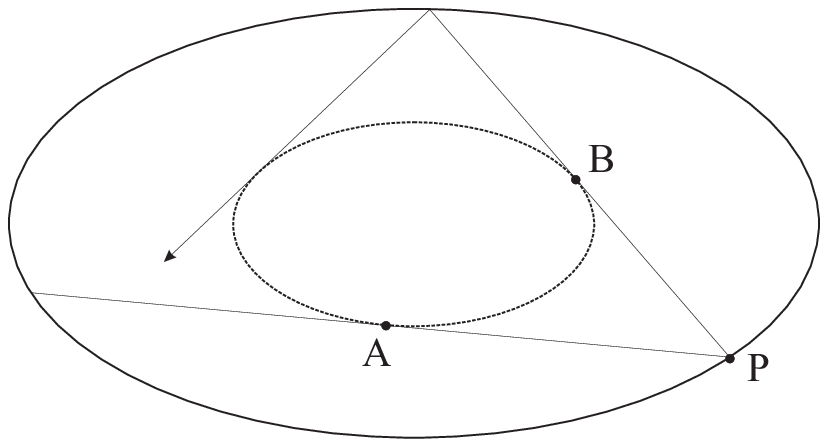,height=1.75in,angle = 0}
\caption{}
\label{figureLazutkin}
\end{center}
\end{figure}

\begin{remark}
Recently, in \cite{SVeselov}, the authors drew a connection between Mather's $\be$-function and Fock's function related to so-called Markov numbers; in particular, they used this relation to answer a question by Fock on  the regularity of this function.
\end{remark}

\vspace{5 pt}

\section{Statement of the main result} \label{secmainresult}
Let us now consider the  framework  {of a standard-like twist map (see Example 2 in Section~\ref{AMtheory})}:
\begin{equation}\label{eq:twist}
T_g(x,y)=(x',y') \ens \text{with} \ens \left\{
\begin{aligned}
x' &=x+y+g(x) \\
y' &=y+g(x)
\end{aligned}
\right.
\end{equation}
with $g$ a 1-periodic, real analytic function of zero mean.  Let $G$
be the primitive of $g$ with zero mean, 
and observe that $G$ is real analytic and 1-periodic as well.  
As a generating function for~$T_g$, we take \[ h(x,x')=\frac{1}{2}(x-x')^2+G(x).\]
As was mentioned earlier, Mather's $\be$-function at any $\om\in\R$ is
defined as the average action of any minimal configuration
$(x_j)_{j\in \ZZ}$ of rotation number~$\om$:
\begin{equation}\label{eq:betaomega}
 \be(\omega)=\lim_{ \substack{N_1\to-\infty \\ N_2 \to +\infty}
%
%
} {\, \frac{1}{N_2-N_1}} \, \sum_{N_1\le j<N_2} h(x_j,x_{j+1}),
\end{equation}
%
and the general theory assures that $\be:\RR \to \RR$ is continuous
everywhere, and is differentiable at any $\omega \in \RR\setminus
\QQ$.
It is worth noting particular symmetry properties in the system at hand:


\begin{lemma}
The function $\omega \mapsto \be(\om)-\frac{1}{2} \om^2$ is $1$-periodic and even on~$\RR$.
\end{lemma}

\begin{proof}
This is a consequence of the following symmetry properties of the generating
function~$h$:
\beglab{eqnsymmproph}
h(x+m,x'+m+1) = h(x,x') + x'-x + \demi, \qquad
h(x',x) = h(x,x') + G(x') - G(x)
\edla
for all $x,x'\in\R$ and $m\in\Z$.
Indeed, take an arbitrary sequence $(x_j)_{j\in \ZZ}$ with a definite
rotation number~$\om$ and consider its finite-segment actions
$A(N_1,N_2) \defeq \sum_{N_1 \le  j  < N_2} h(x_j, x_{j+1})$.
Setting
\[
x^*_j:=x_j+j, \quad x^{**}_j:=x_{-j}
\quad\text{for all $j\in\Z$,}
\]
we get sequences with rotation numbers $\om+1$ and~$-\om$, whose
finite-segement actions can be computed from~\eqref{eqnsymmproph}:
\[
\sum_{N_1 \le  j  < N_2} h(x^*_j, x^*_{j+1})= 
\sum_{N_1 \le  j  < N_2} [h(x_j, x_{j+1}) +x_{j+1}-x_j+\dem] 
= A(N_1,N_2) + x_{N_2} - x_{N_1} + \frac{N_2-N_1}{2}
\]
and, changing the summation index in $\ell=-j-1$,
\begin{multline*}
\sum_{N_1 \le  j  < N_2} h(x^{**}_j, x^{**}_{j+1}) = 
\sum_{-N_2-1<\ell\le -N_1-1} h(x_{\ell+1}, x_{\ell}) = \\[1ex]
\sum_{-N_2\le\ell < -N_1} [h(x_{\ell}, x_{\ell+1}) + G(x_{\ell+1})-G(x_{\ell})] 
= A(-N_2,-N_1) +G(x_{-N_1})-G(x_{-N_2}).
\end{multline*}
Hence, $(x_j)_{j\in\Z}$ is a minimizer  $\iff$
$(x_j^*)_{j\in\Z}$ is a minimizer  $\iff$
$(x_j^{**})_{j\in\Z}$ is a minimizer. Moreover, since $G$ is bounded,
our computation entails
\[
\be(\om+1)=\be(\om)+\om+\frac{1}{2}, \qquad
\be(-\om)=\be(\om)
\]
whence the result follows.
\end{proof}

\smallskip


Our main goal is to show that: {\it if $g$ is not too large (with
  respect to the width of its analyticity strip), then the restriction
  of $\be$ to a suitable subset of Diophantine frequencies is even
  more regular, in the sense that this restriction admits a
  $\gC^1$-holomorphic extension~$\bc$ defined on a complex
  domain {(see below for the definition of $\gC^1$-holomorphic functions)}.}\\

\smallskip


In order to be more precise we need to fix some notation. 
Let us fix once for all $\tau>0$  and consider for $M > 
2\ze(1+\tau)$ (here $\zeta$ is Riemann's zeta function) the following Diophantine set
\begin{equation}	\label{eq:AMR}
A_M^\RR = \left\{\, \om \in \RR \mid
\forall (n,m)\in\ZZ\times\NN^*,\; 
|\om - \frac{n}{m} | \ge \frac{1}{M m^{2+\tau}} \,\right\}.
\end{equation}
This is a closed
subset of the real line, of positive measure, which has empty interior and is
invariant by the integer translations. We also consider the following subset of the complex plane
\begin{equation}	\label{eq:AMC}	
A_M^\CC = \bigl\{\, \om \in\CC \mid\; \exists \om_*\in A_M^\RR
\;\text{ such that }\; |\IM \om| \ge |\om_* - \RE \om| \,\bigr\}
\end{equation}
which has the property that $A_M^\CC \cap \RR= A_M^\RR$ {(see Figure~\ref{fig:one})}.
%
%
%
\begin{figure}  

\begin{center}


\epsfig{file=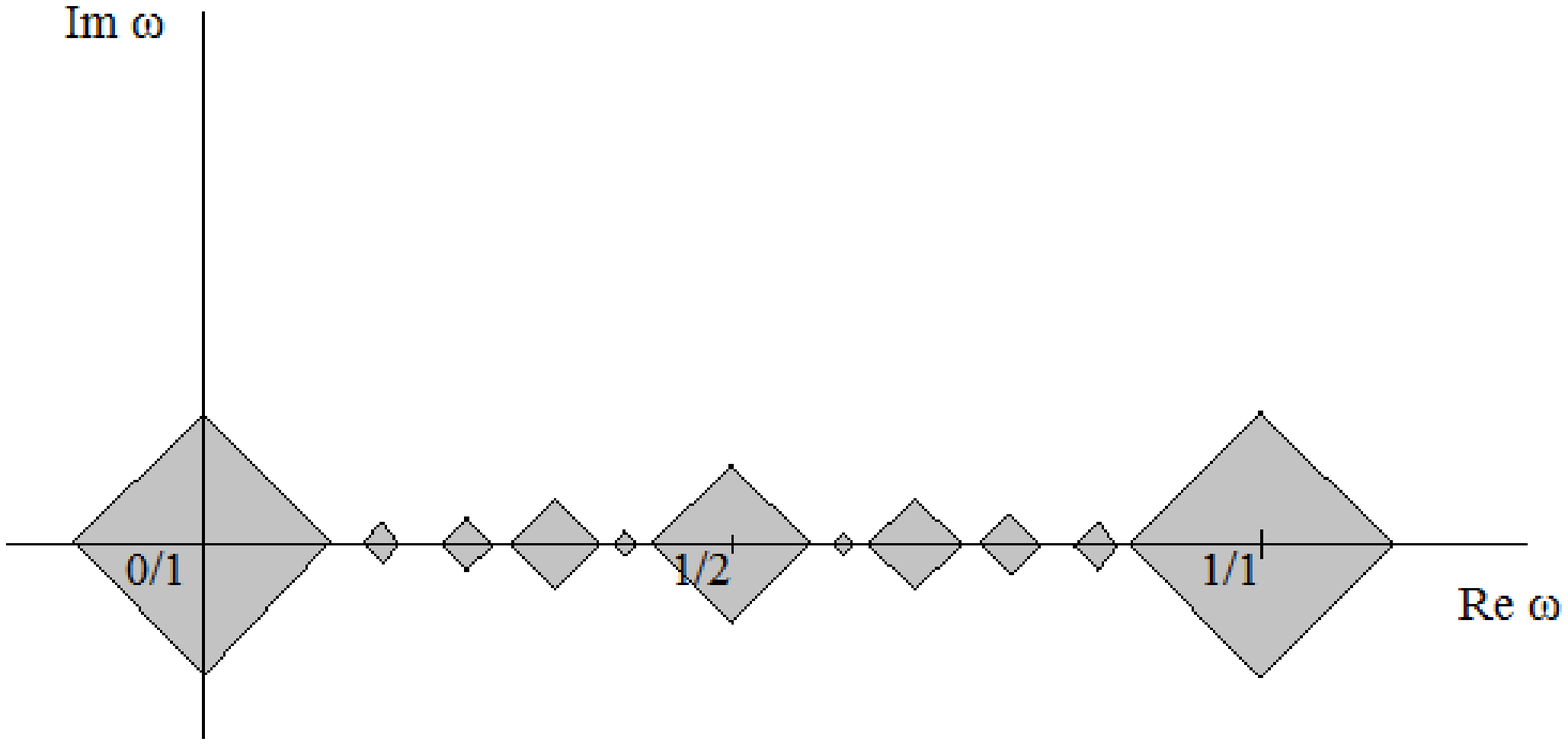,height=1.8in,angle = 0}

\end{center}

\caption{\label{fig:one} {The perfect subset $A_M^\CC\subset\CC$}}

\end{figure}
%
%
Many of the functions that will be important for us satisfy the periodicity condition
$\phi(\om+1)=\phi(\om)$, in fact they  can be even expressed as $\phi=\psi \circ E$, where
\begin{equation} \label{eq:defEexptwopi}
E(\om):=e^{2\pi i \om}
\end{equation}
and $\psi$ is defined on the following compact subset of the Riemann
sphere~$\PP$
(see Figure~\ref{fig:two}):
\beglab{eqdefKMPP}
K_M:=E(A_M^\CC)\cup\{0,\infty\}.
\edla



%
\begin{figure} 

\begin{center}



\epsfig{file=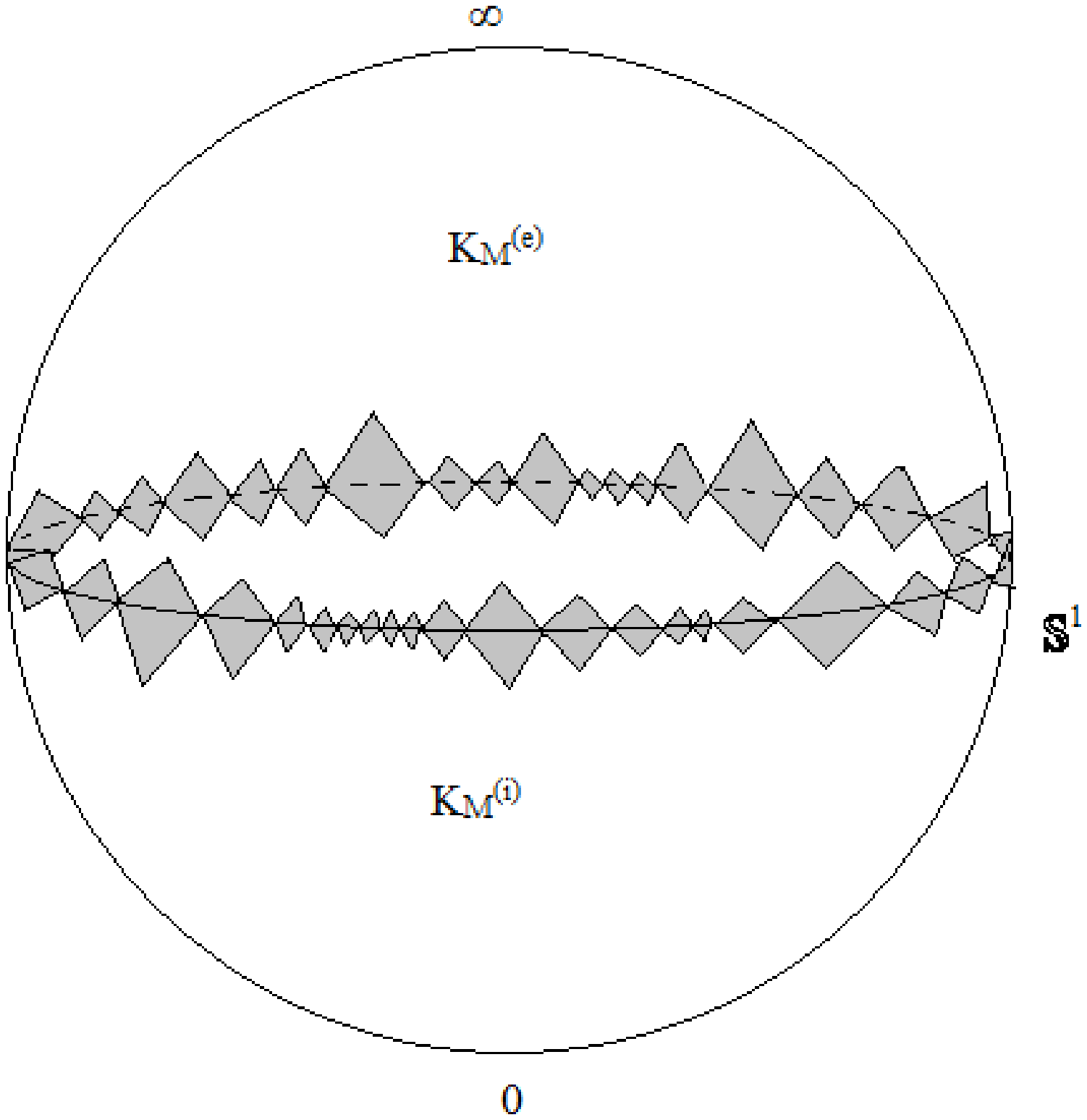,height=3.6in,angle = 0}

\end{center}

\caption{\label{fig:two}{The perfect subset $K_M = K_M\IN \cup K_M\EX\subset\PP$}}

\end{figure}
%

Let us now recall the definition of the spaces of bounded $\gC^1$-holomorphic
functions $\Chol(C,B)$, where $C\subset\CC$ is
perfect and closed and $B$ is a Banach space, and $\Chol(K,B)$, where $K$ is a compact and perfect
subset of~$\PP$.
Both $\Chol(C,B)$ and $\Chol(K,B)$ are Banach spaces, stable under
multiplication if~$B$ is a Banach algebra.

The Banach space
$\Chol(C,B)$ and its norm are defined as follows:
a function $\psi \col C \to B$ is in $\Chol(C,B)$ if it is
continuous and bounded, and there is a bounded continuous function from~$C$ to~$B$,
which we denote by~$\psi'$, such that the function
$\Om\psi \col C\times C \to B$
defined by the formula
\beglab{eqdefOmpsi}
\Om\psi(q,q') \defeq 
\begin{dcases}
\ens\quad \psi'(q) & \text{if $q=q'$,} \\
\frac{\psi(q')-\psi(q)}{q'-q} & \text{if $q\neq q'$,}
\end{dcases}
\edla
is continuous and bounded; the function~$\psi'$ is then
unique\footnote{%
Moreover, for any interior point~$q_0$ of~$C$,
the complex derivative of~$\psi$ at~$q_0$ exists and coincides with
$\psi'(q_0)$.
} 
and we set
\beglab{eqdefnormCholC}
\norm{\psi}_{\Chol(C,B)} \defeq \max \Big\{
\sup_{q\in C} \norm{\psi(q)}_B, \;
%
%
\sup_{(q,q')\in C\times C} \norm{\Om\psi(q,q')}_B
\Big\}.
\edla
This is a Banach space norm equivalent to the one indicated in \cite{He} or \cite{MarmiS}
(or to the one indicated in \cite{CMS},
which is designed to be a Banach algebra norm whenever~$B$ is a Banach
algebra
).

Now, if $K$ is a compact set in $\PP$, we will denote by $\gO(K,B)$
the uniform algebra of continuous functions $\ph\col K\to B$ which are
holomorphic in the interior of $K$, endowed with the norm
\beglab{eqdefnormgOK}
\norm{\ph}_{\gO(K,B)} \defeq \max_{q\in K} \norm{\ph(q)}_B. 
\edla
%
%
To define $\Chol(K,B)$, we assume furhtermore that $K$ is perfect so as to ensure the
uniqueness of the derivative.
Following \cite{FMS}, we cover~$\PP$ with two charts, using~$q$ as a complex coordinate
in~$\CC$ and $\xi = \frac{1}{q}$ in $\PP\setminus\{0\}$;
a function $\ph\col K \to B$ belongs to $\Chol(K,B)$ if
its restriction $\ph_{| K\cap\CC}$ belongs to $\Chol(K\cap\CC,B)$
and the function $\check{\ph} \col \xi \mapsto \ph(1/\xi)$ belongs to 
$\Chol(\check{K} ,B)$, where $\check{K} \defeq \{\, \xi\in\CC \mid 1/\xi \in K
\,\}$ (with the convention $1/0=\infty$),
and we set 
\beglab{eqdefnormCholK}
\norm{\ph}_{\Chol(K,B)} \defeq \max\Big\{
\norm{\ph_{| K\cap\CC}}_{\Chol(K\cap\CC,B)},
\; \norm{\check{\ph}}_{\Chol(\check{K},B)} \Big\},
\edla

As usual, we simply denote by $\gO(K)$ and $\Chol(K)$ the spaces
obtained when $B=\CC$.
The following lemma, 
whose proof is deferred to the appendix,
will be used several times:

\begin{lemma}\label{l:xxx}
  Let $B$ be a Banach space, $A\subset \CC$ be a closed set, and let $K$ be the closure of $E(A)$ in the
  Riemann sphere $\PP$ {with~$E$ as in~\eqref{eq:defEexptwopi}. 
  If $\psi \in \Chol(K, B)$ then the function
  $\psi\circ E \in \Chol( A, B)$, and
  $\|\psi \circ E\|_{\Chol( A, B)} \leq C \|\psi \|_{\Chol(K,B)}$} ($C=2\pi e^{2\pi}$ will do).
\end{lemma}
%


\smallskip

We also define, for any positive real~$R$,
\begin{equation}	\label{eq:SR}
S_R = \{ z\in\CC/\ZZ \mid |\IM z| < R \}
\end{equation}
and $\|\varphi\|_{R} \defeq \sup_{z\in S_R} |\varphi(z)|$ for any
function $\varphi\col S_R \to \C$.
Our main result is:\\


\begin{theorem}\label{th:main}
Let~$R_1$ be positive real. Then there is
    $c=c(\tau, R_1)>0$ such that, for any real analytic 1-periodic
    function~$g$ which has zero mean and extends holomorphically
    to~$S_{R_1}$ with $\|g\|_{R_1} < c$, and for any~$M$ such that
  $1 < \frac{M}{2 \zeta(1+\tau)} < \big( \frac{c}{\|g\|_{R_1}}
  \big)^{1/8}$,
Mather's $\be$-function for the system~\eqref{eq:twist} satisfies
the following:
%
%
%
$\be|_{A_M^\RR}$ admits a complex extension to~$A_M^\CC$ of the form
\[
\bc(\om) \defeq \frac{\om^2}{2} + \Phi^\C_\be(\om),
\]
where $\Phi^\C_\be \in \Chol(A_M^\CC)$.
{Moreover,} 
\begin{itemize}
\item[(i)]  
the derivative of~$\bc$ is an extension of the derivative
  of $\be|_{A_M^\RR}$;
\item[(ii)] the function $\Phi_\be^\C$ is even and 1-periodic, and $\ov{\Phi^\C_\be(\om)}=\Phi^\C_\be(\ov{\om})$;
\item[(iii)] $\Phi^\C_\be=\ti{\Phi} \circ E$ for a function
  $\ti{\Phi} \in \Chol(K_M)$ and $E(z):=e^{2\pi i z}$. This implies that $\Phi^\C_\be$ is defined in an infinite strip $\{\IM \om >\ell\}$ (resp.\ $\{\IM \om <-\ell\}$) and admits a limit as  $ \IM \om \to +\infty$ (resp. $\IM \om \to -\infty$).
\end{itemize} 
\end{theorem}


We thus have
\[
\bc|_{A_M^\RR}=\be|_{A_M^\RR}, \qquad
\bc'|_{A_M^\RR}=\be'|_{A_M^\RR}.
\]
We may refer to~$\bc$ as a $\Chol$-holomorphic function, but notice
that~$\bc$ is not bounded, it is $\bc(\om)-\frac{\om^2}{2}$ that belongs
to $\Chol(A_M^\CC)$.\\

The proof of Theorem \ref{th:main} {is given in
  Sections~\ref{sec:intermed}--\ref{sec:pfThmmain}. It} relies on a
result of \cite{CMS}, which studies regularity properties of the
parametrized KAM curves: the result on the beta function will be
obtained by averaging on the these curves, as we explain below. \\


The extension~$\bc$ of $\be|_{A_M^\RR}$ provided by
Theorem~\ref{th:main} is unique and does not depend on~$M$. This
follows from the quasi-analyticity property established in~\cite{MS2},
according to which the space of functions $\Chol(A^\CC_M)$ is
$\gH^1$-quasi-analytic, {where $\gH^1$ denotes the $1$-dimensional Hausdorff measure}
:
any subset $\Om\subset A^\CC_M$ of positive $\gH^1$-measure is a
uniqueness set\footnote{Namely, a function of this space which vanishes
  identically on~$\Om$ must vanish identically on the whole of~$A^\CC_M$.} for this space of functions.

This quasi-analyticity property has the following striking consequence
on the real Mather's $\be$-function:
\begin{corollary}\label{cor:betaquasianalytic}
Let $R_1>0$ and let~$g$ be real analytic 1-periodic, which has zero mean and extends holomorphically
    to~$S_{R_1}$ so that $\|g\|_{R_1} < c/3$, with $c=c(\tau,R_1)$ as
    in Theorem~\ref{th:main}.
Then there exists $M>2\zeta(1+\tau)$ such that, for every $\om_0\in\RR$,
the function $\be|_{A^\RR_M}$ is determined by the restriction of~$\be$ to any subset
of $[\om_0,\om_0+1]$ of Lebesgue measure $\ge \big( \frac{3\|g\|_{R_1}}{c} \big)^{1/8}$.
One can take $M \defeq 2\zeta(1+\tau) 
\big(\frac{c}{2\|g\|_{R_1}} \big)^{1/8}$.
\end{corollary}

\begin{proof}[Proof of Corollary~\ref{cor:betaquasianalytic}]
Since $\|g\|_{R_1} < c/3$, we get
\[
1 < \big(\tfrac{c}{3\|g\|_{R_1}} \big)^{1/8} < \tfrac{M}{2\zeta(1+\tau)}
< \big(\tfrac{c}{\|g\|_{R_1}} \big)^{1/8}
\]
and we can apply Theorem~\ref{th:main}.
We get a function $\bc(\om) = \frac{\om^2}{2}+\Phi^\C_\be(\om)$ with $\Phi^\C_\be\in
\Chol(A^\CC_M)$.

\smallskip

Let us denote by~$m$ the Lebesgue measure on~$\R$.
Let $\Om \subset [\om_0,\om_0+1]$ have $m(\Om) \ge \big( \frac{3\|g\|_{R_1}}{c} \big)^{1/8}$.
We will prove that $\Om\cap A^\RR_M$ is a uniqueness set for $\Chol(A^\CC_M)$.

\smallskip

As is well known, $m\big( [\om_0,\om_0+1] \setminus A^\RR_M \big) <
2\zeta(1+\tau)/M$, hence
\[
m\big( [\om_0,\om_0+1] \setminus A^\RR_M \big) < 
\big(\tfrac{2\|g\|_{R_1}}{c} \big)^{1/8} < m(\Om).
\]
Consequently, $m(\Om \cap A^\RR_M) = 
m(\Om) - m\big( \Om\cap ( [\om_0,\om_0+1] \setminus A^\RR_M ) \big)
>0$
and $\Om \cap A^\RR_M$ is thus a uniqueness set for $\Chol(A^\CC_M)$.
It follows that $\Phi^\C_\be$ is determined by ${\Phi^\C_\be}_{|\Om \cap A^\RR_M}$;
hence $\bc$, and also $\be|_{A^\RR_M}=\bc|_{A^\RR_M}$, are determined by $\be|_{\Om \cap A^\RR_M}$.
\end{proof}

\smallskip


\section{Intermediate results} \label{sec:intermed}

In order to prove Theorem~\ref{th:main}, let us  first recall
  part of the results of \cite{CMS}.

\smallskip

A {\em parametrized invariant curve} of rotation number $\omega$ for $T_g$ is a pair of continuous functions
$(U,V): \TT \to \TT \times \RR$ such  that 
\begin{equation}\label{eq:conj}
T_g(U(\theta),V(\theta))=(U(\theta+\omega),V(\theta+\omega)) 
\ens \text{for all $\theta \in \TT$.}
\end{equation}
Note that, if $(U,V)$ is a parametrized invariant
curve for $T_{g}$ of rotation number $\om \in \RR\setminus \QQ$,
then  $(U(j\om))_{j\in \ZZ}$ is a minimal configuration of rotation
number $\om$ and the limit in equation \eqref{eq:betaomega} becomes 
\begin{multline}\label{eq:averaging}
\be(\om) = \lim_{\substack{ 
N_1 \to -\infty\\ 
N_2 \to +\infty 
}} 
\, \frac{1}{N_2-N_1} \,
\sum_{N_1\le j < N_2} \Big[\demi
                               \abs*{V\big((j+1)\om\big)}^2+G(U(j\om))\Big] \\
= \demi \int_{\TT} |V(\theta)|^2 d\theta +\int_{\TT} G(U(\theta)) d\theta,
\end{multline}
where we have used Birkhoff's ergodic theorem for the (uniquely)
ergodic rotation of angle $\om \in \RR\setminus \QQ$ on $\TT$.

Since we will be interested in a {perturbative} result (\ie valid for $\|g\|_{R_0}$ small), it is natural to write $U(\theta)=\theta +u(\theta)$, $V(\theta)=\om +v(\theta)$. Taking into account the fact that equation \eqref{eq:twist} implies $x'-x=y'$, we can reduce the quest of an invariant curve to the solution of the following system of equations:
\begin{equation}
\label{eq:u2}
\left\{ \begin{aligned}
& u(\th+\om) - 2 u(\th) + u(\th-\om) = g\big( \th + u(\th) \big) 
\\
& v(\th) = u(\th) - u(\th-\om). 
\end{aligned}
\right.
\end{equation}
It is in fact sufficient to solve the first equation for~$u$:
any $1$-periodic solution~$u$ to this second-order difference
equation 
is the first component of an invariant curve of
frequency~$\om$. 

{Let us denote by $H^\infty(S_R)$ the Banach space of 1-periodic
bounded holomorphic functions on~$S_R$ endowed with the supremum norm $\|\,.\,\|_R$.}
The approach of \cite{CMS} considers the unknown $u=u(\th,\om)$ in equation~(\ref{eq:u2}a) as a 
function of two {complex} variables, {the angle $\th\in S_R$ and
the frequency $\om \in A^\CC_M$, or more precisely as a function
of $\om \in A^\CC_M$ with values in $H^\infty(S_R)$.} We quote the
result as follows:

\begin{theorem}[Theorem 1, \cite{CMS}]\label{thm:uniform}
  Suppose $0<R<R_0$ and $K>0$.  Then there is $c_0=c_0(\tau, K, R, R_0)$ such
  that for any $f \col \RR \to \RR$ 1-periodic with zero mean which
  extends holomorphically to a neighbourhood of~$\ov{S_{R_0}}$ 
  with $\max \{ \|f\|_{R_0}, \|f''\|_{R_0}\} \leq K$, 
for all
  $M> 2 \zeta(1+\tau)$, and for all positive $\varepsilon< c_0M^{-8}$,
  there exists
  $\ti{u}=\ti{u}_{\varepsilon,M}\in \Chol(K_M,H^\infty(S_R))$
  with zero mean such that $u:=\tu \circ E$ (where
  $E(z):=e^{2\pi i z}$) satisfies
\begin{equation}\label{eq:CMS}  
u(\th+\om , \om) - 2 u(\th, \om) + u(\th-\om, \om) = \varepsilon f\big( \th + u(\th, \om) \big)
\end{equation}
for all $\theta \in S_R$ and $\om \in A^\CC_M$ such that $\theta \pm \om \in S_R$,
{and $u(\th,\om)\in\R$ if $\th\in\R/\Z$ and $\om\in A^\RR_M$.}
Moreover $\|\tu\|_{\Chol(K_M,H^\infty(S_R))} \leq \frac{R_0-R}{4}$.
\end{theorem}

\begin{remark}
Actually the statement above differs from the one in \cite{CMS} for a couple of minor aspects. Indeed, in  \cite{CMS} the function $\ti{u}$ is thought as an element of the space
$\Chol(A^\CC_M,H^\infty(S_R\times \mathbb{D}_\rho))$, with $\rho=c_0M^{-8}$, while  here we are only using the result for fixed~$\varepsilon$.

Moreover in the statement of Theorem 1 in \cite{CMS} the constant $c_0$ depends on~$f$. {However,}
analysing the proof one realizes that, for the iterative
scheme to work,  the constant $c_0$ can be determined  only {in terms of}
 $\|f\|_{R_0}$ and $\|f''\|_{R_0}$, 
{and does not actually depend on the \emph{specific}} choice of~$f$
(see  in particular the remark in \cite{CMS} on p.~2053, a few lines before \textsection~4.2).
The last estimate in Theorem~\ref{thm:uniform} does not appear in the
statement in \cite{CMS}, but is a by-product\footnote{In \cite{CMS} the authors use the notation $\|\tu\|_R$ rather than  $\|\tu\|_{\Chol(K_M,H^\infty(S_R))}$} of the proof of Lemma~19 in \cite{CMS}, on p.~2057.
\end{remark}

Let us rephrase the result getting rid of the parameter $\varepsilon$:

\begin{corollary}\label{cor:uxg}
Suppose $0<R<R_1$. 
Then there is $c_1=c_1(\tau, R, R_1)$ such that for any $M> 2
\zeta(1+\tau)$,  and for any  $g:  \RR \to \RR$ 1-periodic with zero
mean which extends holomorphically {to a neighbourhood of~$\ov{S_{R_1}}$} with  $ \|g\|_{R_1} < c_1M^{-8} $,
there exists $\ti{u}=\ti{u}_{M}\in \Chol(K_M,H^\infty(S_R))$ with zero mean, such that 
$u:=\tu \circ E$  (where $E(\om):= e^{2\pi i\om}$) satisfies
\begin{equation}\label{eq:CMSg}
u(\th+\om , \om) - 2 u(\th, \om) + u(\th-\om, \om) = g\big( \th + u(\th), \om \big)
\end{equation}
for all $\theta \in S_R$ and $\om \in A^\CC_M$ such that $\theta
\pm \om \in S_R$,
{and $u(\th,\om)\in\R$ if $\th\in\R/\Z$ and $\om\in A^\RR_M$.}
\end{corollary}
\begin{proof}
Let $R_0:=\frac{R_1+R}{2}$ and $K:=\max\{1, \frac{2}{\pi(R_0-R)^2}
\}$. Cauchy inequalities yield
$\|g''\|_{R_0}\leq\frac{2}{\pi(R_0-R)^2} \|g\|_{R_1}$, therefore 
$$\max \{ \|g\|_{R_0}, \|g''\|_{R_0}\} \leq K \|g\|_{R_1}.$$

Let us set $c_1:= c_0/2$  (for $c_0=c_0(\tau, K, R, R_0)$ as in Theorem~\ref{thm:uniform}), and note that  $f:=\frac{M^8}{c_1}g$ is such that
$$\max \{ \|f\|_{R_0}, \|f''\|_{R_0}\} \leq \frac{M^8}{c_1}\max \{ \|g\|_{R_0}, \|g''\|_{R_0}\} \leq  \frac{M^8}{c_1} K \|g\|_{R_1}\leq K.$$
Therefore, choosing $\varepsilon=c_1 M^{-8}$ and $g=\varepsilon f$, Corollary~\ref{cor:uxg}
immediately follows from Theorem~\ref{thm:uniform}.
\end{proof}

\begin{remark}
From the definition of the function spaces in \cite{CMS} we deduce that not only $\ti{u} \in \Chol(K_M,H^\infty(S_R))$, but $\ti{u}$ admits a normally convergent Fourier expansion
\begin{equation}\label{eq:fou}
\ti{u}(q, \cdot) =\sum_k \hat{u}_k(q) e_k \ \ \mbox{ with }\ \  \
\begin{cases}
\hat{u}_k \in \Chol(K_M) \\ 
e_k(\theta):=e^{2\pi i k\theta} 
\end{cases}
\end{equation}
Moreover (\cf\ \cite{CMS}, Definition 3.2) also
\begin{equation}\label{eq:strip}
\sum_k q^k\hat{u}_k(q) e_k \ \ \mbox{ and } \ \ \sum_k q^{-k} \hat{u}_k(q) e_k
\end{equation}
converge normally in $\Chol(K_M,H^\infty(S_R))$ and (\cf\ \cite{CMS}, Definition 3.3)
\begin{equation}\label{eq:conjugate}
\ov{\hat{u}_k(q)}=\hat{u}_{-k}(1/\ov{q})
\end{equation}
\end{remark} 

\begin{lemma}   \label{lem:upties}
The function $u = \tu \circ E$ of Corollary~\ref{cor:uxg} is
  $1$-periodic in $\om$, it belongs to the space
  $\Chol(A^\CC_M,H^\infty(S_R))$, and it satisfies
$$u(\theta, -\om)=u(\theta, \om), \qquad \ov{u(\theta,\om)}=u(\ov{\theta}, \ov{\om}).$$
\end{lemma}

\begin{proof}
The periodicity of~$u$ follows from the periodicity of~$E(\om)=e^{2\pi i \om}$ and its
$\gC^1$-holomorphy from Lemma~\ref{l:xxx}. 
By construction, $\om \in A^\CC_M \iff -\om \in A^\CC_M$, so setting
$u^*\thom \defeq u(\theta, -\om)$, it is easy to check that
$u^* \in \Chol(A^\CC_M, H^\infty(S_R))$.
Now, $u^*$ is clearly a solution to~\eqref{eq:u2}. Thus, by the
uniqueness argument of \cite{CMS} (see footnote 6 on p.~2038), 
we get $u_{|A^\RR_M}=u^*_{|A^\RR_M}$, hence, by the quasi-analyticity
argument of \cite{MS2}, $u=u^*$.
From~\eqref{eq:conjugate}, it follows that
$\ov{u(\theta,\om)}=
u(\ov{\theta}, \ov{\om})$.
\end{proof}

\smallskip

\section{Proof of Theorem~\ref{th:main}}   \label{sec:pfThmmain}

{We now give ourselves $R_1>0$ and define $c \defeq \big(2
\zeta(1+\tau)\big)^{-8} c_1$, with
$R\defeq R_1/2$ and $c_1=c_1(\tau,R,R_1)$ as in
Corollary~\ref{cor:uxg}.
We suppose that~$g$ and~$M$ satisfy the assumptions of
Theorem~\ref{th:main} with this value of~$c$.
We must find a function~$\bc$ satisfying all the claims of Theorem~\ref{th:main}.}

{Among our assumptions, we have 
  $1 < \frac{M}{2 \zeta(1+\tau)} < \big( \frac{c}{\|g\|_{R_1}}
  \big)^{1/8}$,
therefore $M>2 \zeta(1+\tau)$ and 
$\frac{\|g\|_{R_1}}{c} < \big(2\zeta(1+\tau)\big)^{8} M^{-8}$, 
whence $\|g\|_{R_1} < c_1 M^{-8}$.
We can thus apply Corollary~\ref{cor:uxg} and use the function $u = \tu\circ E$
satisfying equation~\eqref{eq:CMSg} as well as the properties
described in Lemma~\ref{lem:upties}.}

\smallskip

From now on, if $\ti{\varphi}\in \Chol(K_M,H^\infty(S_r))$ has Fourier expansion $\ti{\varphi}(\theta, q)=\sum_k \hat\ph_k(q)e_k(\theta)$ we define
$$\ti{\varphi}^\pm(\theta, q)=\sum_k q^{\pm k}\hat\ph_k(q)e_k(\theta).$$
Note that, by~\eqref{eq:strip}, $\ti{u}^\pm$ both belong to $\Chol(K_M, H^\infty (S_r))$.
Moreover, if $\varphi := \ti{\varphi} \circ E \in \Chol(A^\CC_M, H^\infty (S_r))$, by a slight abuse of notation  we denote
by~$\varphi^\pm=\ti{\varphi}^\pm\circ E$, which boils down to
$\varphi^\pm\thom=\varphi(\theta\pm\om, \om)$.  
Moreover 
we set
\[
v \defeq u -u^-, \quad
U(\th, \om) \defeq \theta + u\thom, \quad 
V\thom \defeq \om + v\thom.
\]
Since $\ti{u}^\pm \in \Chol(K_M, H^\infty (S_r))$ we get that
 $u^\pm$ both belong to $\Chol(A^\CC_M, H^\infty (S_r))$, and the same is true for $v$.
\begin{lemma}\label{L:betilda}
The formula
\begin{equation}\label{eq:betilda}
\begin{split}
\bc(\om)&:= \int_0^1\Big[\frac{1}{2} V(\theta,\om)^2+G\big(U\thom\big)\Big] d\theta \\
&=\frac{1}{2}\om^2+\int_0^1\Big[\frac{1}{2} v(\theta,\om)^2+G(\theta+u(\theta,\om))\Big] d\theta
\end{split}
\end{equation}
defines a function~$\bc$ which can be written in the form
$\bc(\om) = \frac{\om^2}{2} + \Phi_\be^\C(\om)$ with $\Phi_\be^\C  \in
\Chol(A^\CC_M)$; 
in fact, $\Phi_\be^\C = \ti\Phi \circ E$ with $\ti\Phi\in \Chol(K_M)$.
Moreover,  
\begin{equation}\label{eq:biprime}
\frac{d\bc}{d\om}=\int_0^1 V(\theta,\om)\partial_\theta U(\theta, \om) d\theta =\om + \ioi v(\th,\om) \partial_\th u\thom d\theta.
\end{equation}

\end{lemma}
\begin{proof}
  By periodicity of $u$ we immediately get that $\ioi v d\theta =0$,
  $\ioi V d\th = \om$, so the two expressions for $\bc$ above are
  equivalent. 

  The fact that 
\beglab{eqdefPhibeC}
\Phi_\be^\C(\om) \defeq \bc(\om) - \frac{\om^2}{2} =
\int_0^1\Big[\frac{1}{2} v(\theta,\om)^2+G(\theta+u(\theta,\om))\Big] d\theta
\edla
  belongs to $\Chol(A^\CC_M)$ is a consequence of the results in
  \cite{CMS}. Indeed, since $v\in \Chol(A^\CC_M, H^\infty (S_r))$, its
  square also belongs to that space,
%
and Lemma~11 of \cite{CMS} ensures that
\beglab{eqGcircidtiu}
G\circ (id+\ti{u})\in \Chol(K_M,H^{\infty}(S_R)), 
\edla
hence, by Lemma~\ref{l:xxx},
  $G\circ(id+u)= G\circ (id+\ti{u})\circ E \in
  \Chol(A^\CC_M,H^{\infty}(S_R))$. 
On
  the other hand, Lemma~4 of \cite{CMS} ensures that if
\beglab{eqintegrChol}
\phi \in \Chol(A^\CC_M,H^\infty(S_R)) \Imp
[\om \mapsto \ioi \phi\thom d\theta]\in \Chol(A^\CC_M).
\edla


Now we can write
$$\bc(\om)=\ioi \frac{1}{2} V^2 d\th + \ioi G \circ U d\th = \ioi \frac{1}{2} (V^+)^2 d\th + \ioi G \circ U d\th$$ thus
\begin{equation}\label{eq:trick}
\frac{d\bc}{d\om}=\ioi V^+ \partial_\om V^+ d\th + \ioi (g\circ U)\partial_\om U d\th.
\end{equation}
Since $V^+\thom =U(\th+\om,\om)-U\thom$ we get that 
\begin{equation}\label{eq:domv}
\partial_\om V^+=(\partial_\th U )^+ +(\partial_\om U)^+ - \partial_\om U
\end{equation} 
Moreover by \eqref{eq:CMSg} we get that $g\circ U =u^+-2u+u^-=V^+-V$ so that
\begin{equation}\label{eq:gou}
\ioi (g\circ U)\partial_\om U d\th=\ioi (V^+-V)\partial_\om U d\th
\end{equation}  thus
\begin{equation}\label{eq:final}
\begin{split}
\frac{d\ti\be}{d\om}&=\ioi V^+ (\partial_\th U)^+ d\th + \ioi V^+[( \partial_\om U)^+ -\partial_\om U] d\th  + \ioi (g\circ U)\partial_\om U d\th\\
&= \ioi V^+ (\partial_\th U)^+ d\th = \ioi V \partial_\th U d\th
\end{split}
\end{equation}
where the first equality follows from equation \eqref{eq:domv} while equation \eqref{eq:gou} allows us to pass from the first line to the second; and  translation invariance has been used several times as well.

We have $\Phi_\be^\C= \ti\Phi \circ E$ with
\beglab{eqdeftiPhi}
\ti\Phi(q):=\ioi \Big[ \demi \big(\ti{u}(\theta,q)-
\ti{u}^-(\theta,q)\big)^2+G\big(\theta+\ti{u}(\theta,q)\big) \Big] d\theta.
\edla
Using $\ti u, \ti u^- \in\Chol(K_M,H^\infty(S_R))$, the stability of this space under
multiplication, and \eqref{eqGcircidtiu}--\eqref{eqintegrChol}, we
obtain $\ti\Phi \in \Chol(K_M)$.
\end{proof}

\begin{proposition}\label{P:tb=b}
The function $\bc$ defined in Lemma~\ref{L:betilda} coincides with Mather's $\be$-function on the real line. In fact,
\begin{equation}\label{eq:tb=b}
(i)  \ \bc|_{A^\RR_M}=\be|_{A^\RR_M} \qquad
(ii)  \ \bc'|_{A^\RR_M}=\be'|_{A^\RR_M}
\end{equation}
\end{proposition}
\begin{proof}
For $\om \in A^\RR_M$ the sequence $x_j:=U(j\om,\om)$ defines a minimal configuration $(x_j)_{j\in \ZZ}$ with rotation number $\om$; in fact, setting $y_j:=\om+V(j\om, \om)$ yields $T(x_j,y_j)=(x_{j+1},y_{j+1})$. The proof of (i) then follows from equation \eqref{eq:averaging}.

The proof of (ii) follows from a well known formula (see \cite{Siburg}, Theorem 1.3.7-(4)) which expresses the derivative of Mather's $\be$-function in terms of $U$ and $V$:
$$ \be'(\om)=\ioi V\thom \partial_\th U \thom d\theta$$   
Thus $\be'(\om)=\bc'(\om)$ by equation \eqref{eq:biprime}.
\end{proof}


At this stage, only point~(ii) of Theorem~\ref{th:main} remains to be
proved.
According to Lemma~\ref{lem:upties}, we have
\[
u(\th,-\om) = u(\th,\om) = u(\th,\om+1) = \ov{u(\ov\th,\ov\om)}.
\]
This implies
\[
v(\th,-\om) = u(\th,-\om)-u(\th+\om,-\om) = 
u(\th,\om)-u(\th+\om,\om) = -v(\th+\om,\om)
\]
and $v(\th,\om+1) = v(\th,\om) = \ov{v(\ov\th,\ov\om)}$.
In view of~\eqref{eqdefPhibeC}, this yields
\[
\Phi^\C_\be(-\om) = \Phi^\C_\be(\om) = \Phi^\C_\be(\om+1)
= \ov{\Phi^\C_\be(\ov\om)}
\]
and we are done.


\appendix 

\section{Proof of Lemma~\ref{l:xxx}} \label{appendix}

Let $\psi\in\CK$.
We use the same notations for $\check K$ and~$\check\psi$ as in the
definition of the space $\CK$ given in Section~\ref{secmainresult}.
Notice that 
\[
K\cap\C = E(A) \ens\text{or}\ens E(A)\cup\{0\},
\qquad
\check K = E(-A) \ens\text{or}\ens E(-A)\cup\{0\}
\]
according as $\inf\{Re\,\om \mid \om\in A\}>-\infty$ or not for the former,
and $\sup\{Re\,\om \mid \om\in A\}<+\infty$ or not for the latter.

Let $\ph \defeq \psi \circ E$. Clearly, $\ph$ is bounded and
$\dst \sup_{\om\in A}\norm{\ph(\om)}_B \le \sup_{q\in
  K}\norm{\psi(q)}_B
\le \norm{\psi}_\CK$.
For $(\om,\om')\in A\times A$ with $\om \neq \om'$, we have
\begin{equation}\label{eq:two}
\Om\ph(\om,\om') =
\frac{\psi (q')-\psi (q)}{q'-q}
\frac{q'-q}{\om'-\om} =
\frac{\check{\psi} (\xi')-\check{\psi}(\xi)}{\xi'-\xi} \frac{\xi'-\xi}{\om'-\om}
\end{equation}
with $q \defeq E(\om)$, $q' \defeq E(\om')$, $\xi \defeq E(-\om)$, $\xi' \defeq E(-\om')$.
Letting $\om'$ tend to~$\om$, we get
\[
2\pi i E(\om) \psi'(E(\om)) = - 2\pi i E(-\om) \check\psi'(E(-\om))
\]
and we define both $\Om\ph(\om,\om)$ and $\ph'(\om)$ as this common
value.
This way $\Om\ph$ is continuous on $A\times A$.

Let us write $A=A^+ \cup A^-$, where $A^\pm$ are the overlapping regions
\[
A^+ \defeq \{\om \in A \mid \IM \om >-1\}, \qquad
A^- \defeq \{\om \in A \mid \IM \om <1\}.
\]
If both $\om$ and~$\om'$ belong to $A^+$ (resp.\ $A^-$), then the quantity 
$|\frac{q'-q}{\om'-\om}|$ (resp.\ $|\frac{\xi'-\xi}{\om'-\om}|$) is bounded by
$2\pi e^{2\pi}$, 
hence by the first (resp.\ second) expression in \eqref{eq:two} we get
\beglab{ineqnormOMph}
\norm{\Om\ph(\om, \om')}_B \le 2\pi e^{2\pi} \norm{\psi}_\CK,
\edla
and also $\norm{\ph'(\om)}_B \le 2\pi e^{2\pi} \norm{\psi}_\CK$ by continuity.
If $\om$ and~$\om'$ do not lie in the same region, then 
%
$|\om -\om'|\ge2$, hence 
$\norm{\Om\ph(\om, \om')}_B 
= \norm{\frac{\psi (q')-\psi (q)}{\om'-\om}}_B
\le \norm{\psi}_\CK$.

Therefore, \eqref{ineqnormOMph} always holds true, which completes the proof of our claim.


\vspace{1.truecm}

\end{document}